\documentclass[11pt,a4paper]{amsart}

\usepackage[utf8]{inputenc}
\usepackage{amssymb,amsmath,amsbsy,amscd,amsfonts,amsthm,latexsym}
\usepackage[mathcal]{eucal}
\usepackage{times, euler}
\usepackage{eurosym}

\usepackage{mathtools}

\usepackage{stackrel, enumerate}
\usepackage{mathabx}

\usepackage[usenames]{color}
 
\usepackage[top=3cm,bottom=3cm,left=2cm,right=2cm,footskip=17mm]{geometry}

\usepackage{hyperref}

\usepackage{tikz,pgf}
\usepackage{tikzit}

\tikzstyle{1function}=[fill=white, draw=black, shape=rectangle, minimum width=0.75cm, minimum height=1 cm]
\tikzstyle{2function}=[fill=white, draw=black, shape=rectangle, minimum width=1cm, minimum height=1 cm]
\tikzstyle{3function}=[fill=white, draw=black, shape=rectangle, minimum width=1.5cm, minimum height=1cm]
\tikzstyle{multi function}=[fill=white, draw=black, shape=rectangle, minimum width=5cm, minimum height=1cm]
\tikzstyle{multi function small}=[fill=white, draw=black, shape=rectangle, minimum width=4cm, minimum height=1cm]
\tikzstyle{Multi function smaller}=[fill=white, draw=black, shape=rectangle, minimum width=3.5 cm, minimum height=1 cm]
\tikzstyle{graph node}=[fill=black, draw=black, shape=circle]
\tikzstyle{smallbox}=[fill=white, draw=black, shape=rectangle, minimum width=0.5 cm, minimum height=0.5cm]
\tikzstyle{Circle1}=[fill=white, draw=black, shape=circle, minimum size=26 mm]

\tikzstyle{black dashed}=[-, draw=black, dashed]
\tikzstyle{new edge style 0}=[thick, ->]
\tikzstyle{thickedge}=[thick, -]

\hypersetup{
    colorlinks=true,
    linkcolor=blue,
    filecolor=cyan,      
    urlcolor=blue,
    citecolor=magenta
}

\definecolor{darkgreen}{rgb}{0,0.6,0}

\usepackage{xy}
\xyoption{all}

\numberwithin{equation}{section}

\theoremstyle{plain}
    \newtheorem{theorem}[equation]{Theorem}
    \newtheorem{lemma}[equation]{Lemma}
    \newtheorem{lemma-definition}[equation]{Lemma-Definition}
    \newtheorem{corollary}[equation]{Corollary}
    \newtheorem{proposition}[equation]{Proposition}

\theoremstyle{definition}
    \newtheorem{definition}[equation]{Definition}

    \newtheorem{remark}[equation]{Remark}

\newcommand{\C}{\mathcal{C}}
\newcommand{\N}{\mathbb{N}}
\newcommand{\Z}{\mathbb{Z}}

\newcommand{\F}{\mathbb{F}}

\renewcommand{\H}{\mathcal H}

\renewcommand{\phi}{\varphi}
\renewcommand{\epsilon}{\varepsilon}

\newcommand{\ol}{\overline}

\newcommand{\ot}{\otimes}

\DeclareMathOperator{\Irr}{Irr}
\DeclareMathOperator{\Hom}{Hom}
\DeclareMathOperator{\End}{End}
  
\DeclareMathOperator{\Ker}{Ker}

\DeclareMathOperator{\Aut}{Aut}
\DeclareMathOperator{\Ind}{Ind}

\DeclareMathOperator{\GL}{GL}

\DeclareMathOperator{\Rep}{\mathrm{Rep}}

\DeclareMathOperator{\Mod}{Mod}

\newcommand{\one}{\mathbf{1}}

\newcommand{\Vect}{\operatorname{Vect}}

\newcommand{\modd}{\, \mathrm{mod} \, }

\newcommand{\Ow}{\mathfrak o}	  

\newcommand{\wt}{\widetilde}

\newcommand{\la}{\lambda}
\newcommand{\Id}{\text{Id}}

\renewcommand{\Im}{\text{Im}}

\newcommand{\M}{\text{M}}

\newcommand{\Ab}{\text{Ab}}
\renewcommand{\preceq}{\preccurlyeq}
\renewcommand{\P}{\mathbb{P}}
\newcommand{\Autc}{\text{Aut}_{\mathcal{C}_{+1}}}

\begin{document}

	\title[Functor morphing image ]{The image of functor morphing}


\author{Ehud Meir}
\address{Institute of Mathematics, University of Aberdeen, Fraser Noble Building, Aberdeen AB24 3UE, UK}
  \email{ehud.meir@abdn.ac.uk}

	\begin{abstract}  
Functor morphing provides a method to translate complex representations of automorphism groups of finite modules over finite rings to representations of automorphism groups of functors in some abelian category. In this paper we give an explicit criterion for a representation to be in the image of functor morphing using the action of parabolic subgroups. We then demonstrate this criterion on Borel groups of finite fields. 
	\end{abstract}
	\maketitle

\section{Introduction}
Let $M$ be a finite module over a finite ring $R$. In \cite{CMO4} Tyrone Crisp, Uri Onn, and the author, introduced a new technique called functor morphing to study the representation theory of $\Aut_R(M)$. Assume that $M=M_1^{a_1}\oplus\cdots\oplus M_n^{a_n}$ is the decomposition of $M$ into a direct sum of non-isomorphic indecomposable modules. Functor morphing associates to every irreducible representation $V$ of $\Aut_R(M)$ an additive functor $F:\langle M_1,\ldots, M_n\rangle\to Ab$ and an irreducible representation $\wt{V}$ of $\Aut(F)$. Here $\langle M_1,\ldots, M_n\rangle$ stands for the full additive subcategory of $R-\modd$ generated by $M_1,\ldots, M_n$. 
In most cases the procedure of functor morphing reduces the numerical complexity of the representation $V$. The functor $F$ is the minimal (with respect to the subquotient relation) functor that satisfies $\Hom_{\Aut_R(M)}(KF(M),V)\neq 0$, and the representation $\wt{V}$ is given by $\Hom_{\Aut_R(M)}(KF(M),V)$, where $\Aut(F)$ acts on $KF(M)$ in the natural way. The fact that $F$ is uniquely defined and that $\wt{V}$ is irreducible follows from a construction of the author of symmetric monoidal categories using invariant theory \cite{meir1}. The study of the representation theory of the groups $\Aut_R(M)$ stems from the representation theory of the groups $\GL_n(\Ow/(\pi^l))$, where $\Ow$ is a discrete valuation ring with uniformizer $\pi$ and finite quotient field. One of the main philosophical consequences of functor morphing is the fact that the representation theory of these group contains in fact the representation theory of $\Aut_R(M)$ for very general $R$ and $M$. See also \cite{CMO1,CMO2,CMO3} for a study of the representation theory of $\GL_n(\Ow/(\pi^l))$ using generalizations of Harish-Chandra functors and Clifford Theory.

Functor morphing thus gives a new way to study the representation theory of a very large class of finite groups. 
We get a stratification $\Irr(\Aut_R(M)) \cong \sqcup_F \Irr^{M,fm}(\Aut(F))$,  where $\Irr^{M,fm}(\Aut(F))$ stands for the irreducible representations of $\Aut(F)$ that arise from functor morphing of representations of $\Aut_R(M)$. This raises the question of what irreducible representations actually appear in $\Irr^{M,fm}(\Aut(F))$. The goal of the present paper is to answer this question, and provide concrete examples for the family of upper triangular groups $B_n(\F_q)$. 

To state our main result, we introduce some terminology. Let $M=\bigoplus_i M_i^{a_i}$, $F$, and $\C=\langle M_1,\ldots, M_n\rangle$, be as above. We say that $(X\stackrel{f}{\to} Y)$ is a minimal presentation of $F$ if $$\Hom_{R}(Y,-)\stackrel{f^*}{\to}\Hom_{R}(X,-)\to F\to 0$$ is a minimal resolution of $F$ in the category $\C_{+1} = Fun(\C,\Ab)$. 
We have already seen in \cite{CMO4} the connection between the minimal resolution of $F$ and functor morphing. In Section \ref{sec:lin.rel} we prove the following:
\begin{lemma}[See Lemma \ref{lem:XM}]
A necessary condition for representations of $\Autc(F)$ to appear as functor morphing of representations of $\Aut_R(M)$ is that if $(X\stackrel{f}{\to} Y)$ is a minimal presentation of $F$ then $X$ is a direct summand of $M$.
\end{lemma}

If $G\subseteq F$ is a subfunctor, we write $Z_G$ for the unique object in $\C$ such that $\Hom_R(Z_G,-)$ is the projective cover of $G$. 
We write $P_G\subseteq \Autc(F)$ for the subgroup $\{\phi\in \Autc(F)|\Im(\phi - \Id_F)\subseteq G\}$. 

In Section \ref{sec:relations} we prove our main result:
\begin{theorem}
Let $M$, $\C$, and $F$ be as above. Assume that $F$ has a minimal presentation $(X\stackrel{f}{\to} Y)$ and that $M=X\oplus Z$. 
Let $V$ be an irreducible representation of $\Autc(F)$. Then $V$ is the functor morphing of an irreducible representation of $\Aut_R(M)$ if and only if the following condition holds: for every semisimple subfunctor $G\subseteq F$ such that $Z_G$ is not a direct summand of $Z$, it holds that $\Hom_{P_G}(\one,V)=0$. 
\end{theorem}
The main result has the following corollary, characterizing the situations in which all the irreducible representations of $\Autc(F)$ arise from functor morphing:
\begin{corollary}
Let $M=X\oplus Z$ and let $(X\xrightarrow{f} Y)$ be a minimal resolution for $F$. Then the following conditions are equivalent:
\begin{enumerate}
\item Every irreducible representation of $\Autc(F)$ is the functor morphing of an irreducible representation of $\Aut_R(M)$. 
\item The trivial representation $\one$ of $\Autc(F)$ is the functor morphing of some irreducible representation of $\Aut_R(M)$.
\item If $\Hom(X_1,-)$ is the projective cover of $soc(F)$ then $X_1$ is a direct summand of $Z$. 
\end{enumerate}
\end{corollary}

Finally, in Section \ref{sec:examples} we give examples. We first consider the family of groups $\GL_n(\F_q)$. By \cite{Zelevinsky} we know that these groups give rise to a positive self-adjoint Hopf algebra $\H$. We explain the connection between our main result and the structure of $\H$.  
The second family of examples that we give is of the family of groups $B_n(\F_q)$ of upper triangular matrices over $\F_q$. We show how our result translates to a combinatorial criterion, and we show how this gives a new enumeration of the irreducible representations of $B_n(\F_q)$ when $n\leq 4$. 
A long standing conjecture of Higman from 1960 \cite{Higman1} states that for every $n$ the number of conjugacy classes in the unipotent radical $U_n(\F_q)$ of $B_n(\F_q)$ is a polynomial in $q$ \cite{VA4,So1}. Since the number of conjugacy classes is the same as the number of irreducible representations, and since $U_n(\F_q)$ is a normal subgroup of $B_n(\F_q)$ with an abelian quotient, this suggests that there are deep connections between enumeration of irreducible representations of $B_n(\F_q)$ and Higman's conjecture. It is our hope that, following this paper, the technique of functor morphing will be further used to study this conjecture. 

\section{Preliminaries}
\subsection{Functor morphing}
Let $R$ be a finite ring and let $M$ be an $R$-module. Functor morphing is a tool that was introduced in \cite{CMO4} to study the representation theory of $\Aut_R(M)$. We recall here the details. Write $M=M_1^{a_1}\oplus\cdots\oplus M_n^{a_n}$ for the decomposition of $M$ into non-isomorphic indecomposable modules. Write $\C = \langle M_1,\ldots M_n\rangle\subseteq R-\modd$ for the full subcategory of finite $R$-modules of the form $\bigoplus_{i=1}^n M_i^{b_i}$ and write $\C_{+1} = Fun(\C,Ab)$ for the category of additive functors from $\C$ to the category of finite abelian groups. The isomorphism classes of the objects in $\C_{+1}$ are partially ordered, where $F\preceq G$ if and only if $F$ is isomorphic to a subquotient of $G$. For every $F\in\C_{+1}$ it holds that $F(M)$ is an $\Autc(F)\times\Aut_R(M)$-set in a natural way. As a result, the linearlization $KF(M) = span\{u_v\}_{v\in F(M)}$ is an $\Autc(F)\times\Aut_R(M)$-representation. 
\begin{theorem}[Theorem A, \cite{CMO4}] Let $V$ be an irreducible representation of $Aut_R(M)$. There is a
unique (up to isomorphism) functor $F\in \C_{+1}$ such that $\Hom_{Aut_R(M)}(KF(M),V)\neq 0$ and  $\Hom_{Aut_R(M)}(KG(M),V)= 0$ for every $G\prec F$. 
\end{theorem}
This theorem shows that to every irreducible representation $V$ of $\Aut_R(M)$ we can attach a functor $F\in \C_{+1}$. 
In fact, more is true. Write $\ol{KF(M)}$ for the largest subobject of $KF(M)$ that satisfies $$\Hom_{\Aut_R(M)}(\ol{KF(M)},KG(M))=0$$ for every $G\prec F$. 
The $\Autc(F)$-action on $KF(M)$ restricts to an action on $\ol{KF(M)}$ and we have an algebra morphism 
$$\Phi_F(M):K\Autc(F)\to\End_{\Aut_R(M)}(\ol{KF(M)}).$$
\begin{theorem}[Theorem B, \cite{CMO4}] The map $\Phi_F(M)$ is surjective. If the numbers $a_i$ that appear in the decomposition $M_i = \bigoplus_i M_i^{a_i}$ are big enough, then $\Phi_F(M)$ is an isomorphism.
\end{theorem}
This theorem enables us to attach to $V$ an irreducible representation of $\Autc(F)$. Indeed, we can decompose $\ol{KF(M)}$ as an $\Aut_R(M)$-representation and write  $\ol{KF(M)} = \bigoplus_j U_j\ot V_j$, where $U_j$ are vector spaces and $V_j$ are non-isomorphic $\Aut_R(M)$-representations. The action of $\Autc(F)$ commutes with the action of $\Aut_R(M)$, and is thus given by an action on the vector spaces $U_j$. The fact that $\Phi_F(M)$ is surjective implies that $U_j$ is an irreducible $\Autc(F)$-representation. 
\begin{definition}  We call $(F,U_j^*)$ the functor morphing of $(M,V_j)$, and we write $(M,V_j)\leadsto (F,U_j^*)$. 
\end{definition}
\begin{remark}
\begin{enumerate}
\item By the first theorem we see that for every irreducible representation $V$ of $\Aut_R(M)$ there is a unique (up to isomorphism) functor $F$ such that $V$ appears in $\ol{KF(M)}$. So we can talk freely about \textit{the} functor morphing of $V$. 
\item The reason we take $U_j^*$ instead of $U_j$ in the definition is that we want to make sure that in the special case $F=\Hom_R(M,-)$ it holds that representations functor morph to themselves under the identification $\Autc(F)\cong \Aut_R(M)$. 
\item In the paper \cite{CMO4} the ring $R$ was an $\Ow_l:= \Ow/(\pi^l)$-algebra, where $\Ow$ is some discrete valuation ring with uniformizer $\pi$, and $\C_{+1}$ was the category of $\Ow_l$-linear functors from $\C$ to $\Ow_l-\Mod$. It can easily be seen that this category is equivalent to the category of additive functors from $\C$ to $Ab$. In one direction we can just compose with the forgetful functor $\Ow_l-\Mod\to Ab$, and in the other direction we can use the fact that if $F:\C\to Ab$ is an additive functor, then $F(X)$ has a canonical $\Ow_l$-module structure for every $X\in\C$. In this paper we will just use the definition $\C_{+1}= Fun(\C,Ab)$ that does not require to refer to a discrete valuation ring. 
\end{enumerate}
\end{remark}

\subsection{Hom-spaces}
We begin with the following definition.
\begin{definition}
Let $F,G\in \C_{+1}$ and let $H\subseteq F\oplus G \in\C_{+1}$. We define $$T_H\in \Hom_{\Aut_R(M)}(KF(M),KG(M))$$ $$u_v\mapsto \sum_{\substack{ v'\in G(M) \\ (v,v')\in H(M)}}u_{v'}.$$
\end{definition}
It is easy to see that $T_H$ is indeed a morphism of $\Aut_R(M)$-modules. 
The following is Proposition 2.23 in \cite{CMO4}
\begin{proposition}
The set $\{T_H\}_{H\subseteq F\oplus G}$ spans $\Hom_{\Aut_R(M)}(KF(M),KG(M))$.
\end{proposition}
Since $KF(M)$ is a permutation module it holds that $(KF(M))^*\cong KF(M)$ as $\Aut_R(M)$-representations. 
We thus have a natural isomorphism 
\begin{equation}\label{eq:hom.isos}\Hom_{\Aut_R(M)}(\one, K(F\oplus G)(M))\cong \Hom_{\Aut_R(M)}(\one,KF(M)\ot KG(M))\cong \Hom_{\Aut_R(M)}(KF(M),KG(M)).\end{equation} 
\begin{remark}The trivial one-dimensional representation $\one$ is isomorphic to $K0(M)$, where $0$ is the zero functor from $\C$ to $Ab$. It can easily be seen that under the isomorphism of Equation \ref{eq:hom.isos} the element $T_H\in\Hom_{\Aut_R(M)}(K0(M),K(F\oplus G)(M))$ corresponds to the element $T_H\in \Hom_{\Aut_R(M)}(KF(M),KG(M))$. We will use this in what follows, when we study the linear relations the elements $T_H$ satisfy. 
\end{remark}

The next definition will appear often in our study of functor morphing:
\begin{definition}
Let $F\in \C_{+1}$. We say that $(X\stackrel{f}{\to}Y)$ is a presentation of $F$ if $X,Y\in \C$, and $F$ has a resolution
$$\Hom_R(Y,-)\stackrel{f^*}{\to}\Hom_R(X,-)\to F\to 0.$$ The presentation is called minimal if the resolution is minimal. 
\end{definition}
By Lemma 4.7. in \cite{CMO4} we know that if $F$ has a minimal resolution $(X\stackrel{f}{\to}Y)$ and $\ol{KF(M)}\neq 0$ then $X$ is a direct summand of $M$. We will prove this again in Lemma \ref{lem:XM} below. For more background on homological algebra in the category $\C_{+1}$, see Appendix A in \cite{CMO4}.

\section{linear relations in hom-spaces}\label{sec:lin.rel}
By Equation \ref{eq:hom.isos} we see that it is enough to study the linear relations among the elements $$\{T_H\}_{H\subseteq F} \subseteq \Hom_{\Aut_R(M)}(\one, KF(M))$$ for an arbitrary functor $F\in \C_{+1}$. 
\begin{definition}
Let $M\in \C$ and $F\in \C_{+1}$. We say that $F$ is $M$-redundant if it holds that $$F(M) = \bigcup_{H\subsetneq F} H(M).$$ If $F$ is not $M$-redundant we say that $F$ is $M$-irredundant.  
\end{definition}
\begin{definition}
For every $G\in \C_{+1}$ we write $$X_G(M) = G(M)\backslash \cup_{G'\subsetneq G}G'(M).$$ Thus, $G$ is $M$-redundant if and only if $X_G(M) = \varnothing$. 
\end{definition}
\begin{definition}
For every $H\subseteq F$ we define $L_H = \sum_{v\in X_G(M)}u_v\in \Hom_{\Aut_R(M)}(\one,KF(M))$
\end{definition}
\begin{lemma}\label{lem:XG}
It holds that 
$$H(M) = \bigsqcup_{G\subseteq H} X_G(M).$$
\end{lemma}
\begin{proof}
We first notice that that we have $X_{G_1}(M)\cap X_{G_2}(M)= \varnothing$ for $G_1\neq G_2$. Indeed, if $v\in X_{G_1}(M)$ and $v\in X_{G_2}(M)$ then 
$v\in G_1(M)\cap G_2(M) = (G_1\cap G_2)(M)$. But $G_1\cap G_2\subsetneq G_1,G_2$, which implies that $v\notin X_{G_1}(M)$ and $v\notin X_{G_2}(M)$, a contradiction.

It is clear that $X_G(M)\subseteq H(M)$ for every $G\subseteq H$. 
For $v\in H(M)$, write $Y_v = \{G| v\in G(M)\}$. Then $Y_v$ is ordered by inclusion, and it contains a single minimal element. Indeed, if $G_1$ and $G_2$ are two minimal functors in $Y_v$, then $G_1\cap G_2\in Y_v$ is minimal and $G_1\cap G_2\subseteq G_1,G_2$. 
Take $G$ to be the unique minimal element in $Y_v$. Then it holds that $v\in G(M)\backslash \bigcup_{G'\subsetneq G}G'(M)= X_G(M)$, which implies the lemma.  
\end{proof}
\begin{corollary}\label{cor:TL}
For $H\subseteq F$ it holds that $T_H = \sum_{G\subseteq H} L_G$. 
\end{corollary}
\begin{proof}
Using the equality from the previous lemma we get 
$$T_H = \sum_{v\in H(M)} u_v = \sum_{G\subseteq H}\sum_{v\in X_G(M)}u_v = \sum_{G\subseteq H}L_G.$$
\end{proof}
The last corollary shows us how to write the elements $T_H$ as linear combinations of the elemenets $L_G$.
We can also go the other way around.
\begin{lemma}\label{lem:LT} Let $H$ be a functor, and let $H_1,\ldots, H_r$ be the maximal proper subfunctors of $H$. 
It holds that \begin{equation}\label{eq:LH}L_H = T_H - \sum_i T_{H_i} + \sum_{i_1<i_2} T_{H_{i_1}\cap H_{i_2}} - \sum_{i_1<i_2<i_3}T_{H_{i_1}\cap H_{i_2}\cap H_{i_3}} + \cdots + (-1)^r T_{H_1\cap H_2\cap\cdots\cap H_r}.\end{equation}
\end{lemma}
\begin{proof}
It holds that $X_H(M) = H(M) \backslash \bigcup_{H'\subsetneq H}H'(M) = H(M) \backslash\bigcup_{i=1}^r H_i(M)$. 
The equality now follows from applying the Inclusion-exclusion principle. 
\end{proof}
\begin{remark}\label{rem:ssi}
We could have also phrased the last lemma using all the subfunctors of $H$, instead of just the maximal ones. However, there will be a huge advantage in using only the maximal ones. Indeed, For every $i_1<i_2<\cdots<i_s$ it holds that $H/ (H_{i_1}\cap \cdots \cap H_{i_s})$ is a direct sum of simple functors. This will enable us to reduce some calculations in Section \ref{sec:relations} to the familiar playfield of the groups $\GL_n(\F_q)$, as any automorphism group of a semisimple functor is of the form $\prod_i \GL_{n_i}(\F_{q_i})$. 
\end{remark}
\begin{proposition}\label{prop:LHred}
If $H$ is $M$-redundant then $L_H=0$. The set $\{L_H\}$ for $H$ $M$-irredundant is a basis for $\Hom_{\Aut_R(M)}(\one,KF(M))$. 
\end{proposition}
\begin{proof}
Corollary \ref{cor:TL} shows that we can write the elements $T_H$ as linear combinations of the elements $L_H$. This implies that $\{L_H\}_{H\subseteq F}$ spans $\Hom_{\Aut_R(M)}(\one,KF(M))$. If $H$ is $M$-redundant then $X_H(M)=\varnothing$ and the definition then gives that $L_H=0$. On the other hand, the supports of $L_H$ for a functor $H$ that is $M$-irredundant is exactly $X_H(M)$. Since $X_H(M)\cap X_{H'}(M)=\emptyset$ for $H\neq H'$ by Lemma \ref{lem:XG}, the set $\{L_H\}$ where $H$ is an $M$-irredundant subfunctor of $F$ is linearly independent, and we get a basis for $\Hom_{\Aut_R(M)}(\one,KF(M))$.
\end{proof}

We thus see that the concept of $M$-redundancy is fundamental to understanding the hom-spaces and endomorphism algebras. The next proposition gives an explicit criterion for a functor to be $M$-redundant. 
\begin{proposition}\label{prop:redundancycriterion}
Assume that $F$ has a minimal presentation $(X\stackrel{f}{\to} Y)$. Then $F$ is $M$-redundant if and only if $X$ is not a direct summand of $M$.
\end{proposition}
\begin{proof} 
Write $p:\Hom_R(X,-)\to F$. 
Assume first that $F$ is $M$-irredundant. Take $v\in F(M)\backslash \bigcup_{G\subsetneq F} G(M)$. Take $\phi\in p_M^{-1}(v)$. Then $\phi\in \Hom_R(X,M)$. 
In particular, it induces a natural transformation $\phi^*:\Hom_R(M,-)\to \Hom_R(X,-)$. We then have that $\Im(p\phi^*) = G$ is a subfunctor of $F$. It holds that $p\phi^*(\Id_M) = p(\phi) = v\in G(M)$. But $v$ is not contained in any $H(M)$ for $H$ a proper subfunctor of $F$, so $G=F$. Thus means that $\Im(p\phi^*)=F$, and so $p\phi^*$ is surjective. Since $(X\stackrel{f}{\to}Y)$ is a minimal presentation of $F$, this implies that $\phi^*:\Hom(M,-)\to \Hom_R(X,-)$ is surjective. In particular, $\phi^*:\Hom_R(M,X)\to\Hom_R(X,X)$ is surjective, and so there is  a $\psi:M\to X$ such that $\phi^*(\psi) = \psi\phi = \Id_X$. That is- $X$ is a direct summand of $M$. 

In the other direction, assume that $X$ is a direct summand of $M$ and that $F$ is $M$-redundant. Let $\phi:X\to M$ be a split injection. Then $p(\phi)\in F(M)$, and since $F$ is $M$-redundant, there is a proper subfunctor $G\subsetneq F$ such that $p(\phi)\in G(M)$. 
Let $\Hom_R(Z,-)$ be the projective cover of $G$. We get the following commutative diagram:
$$\xymatrix{ & \Hom_R(Z,-)\ar[r]^-{p'}\ar[d]^{g^*}& G\ar[d]\\ \Hom_R(Y,-)\ar[r]^{f^*} & \Hom_R(X,-)\ar[r]^-p& F}.$$
Let $\psi\in p'^{-1}(p(\phi))$. Then it holds that $p(g^*(\psi))= p(\psi\circ g) = p(\phi)$, which implies that $\psi g - \phi\in \Ker(p_M)$. But $\Ker(p_M)$ contains exactly the maps $X\to M$ that split over $f$, so we can  write $\psi g = \phi + \rho f $ for some $\rho:Y\to M$. Since $\phi$ is a split injection, there is a map $\nu:M\to X$ such that $\nu\phi = \Id_X$. We then get the equality 
$$\nu\psi g = \Id_X + \nu\rho f.$$ Since $(X\stackrel{f}{\to} Y)$ is a minimal presentation of $F$ the map $\nu\rho f$ must be nilpotent, and $\Id_X + \nu\rho f$ is therefore invertible. We then get 
$$(\Id_X + \nu\rho f)^{-1}\nu\psi g = \Id_X.$$ This means that $g:X\to Z$ is a split injection, which means that $g^*$ is surjective. But by diagram chasing this means that the inclusion $G\to F$ is surjective, contradicting the fact that $G$ is a proper subfunctor of $F$. 
\end{proof}
The following corollary follows immediately from the proof of the proposition:
\begin{corollary}
Let $M,F,X$, and $Y$ be as above. Then $X_F(M)= \{p_M(\phi) | \phi:X\to M\text{ is a split injection}\}$.
\end{corollary}

Consider now the space $\End_{\Aut_R(M)}(KF(M))\cong \Hom_{\Aut_R(M)}(\one,KF(M)\ot KF(M))\cong \Hom_{\Aut_R(M)}(\one,K(F\oplus F)(M))$. 
This space is spanned by $\{T_H\}_{H\subseteq F\oplus F}$. The subrepresentation $\ol{KF(M)}$ of $KF(M)$ is characteristic, and so restriction gives an algebra homomorphism $$\Theta_F(M):\End_{\Aut_R(M)}(KF(M))\to \End_{\Aut_R(M)}(\ol{KF(M)}).$$ Since $\Rep(\Aut_R(M))$ is semisimple, this map is surjective. 
\begin{definition} We say that a functor $H\subseteq F\oplus F$ is $F$-tight if it holds that both $pr_1,pr_2:H\to F$ are isomorphisms. Such functors $H$ are in one-to-one correspondence with elements of $\Autc(F)$, where $\alpha\in \Autc(F)$ corresponds to the functor 
$$H_{\alpha} = \Im(F\xrightarrow{\Id_F\oplus \alpha} F\oplus F).$$
\end{definition}
The following lemma is a consequence of Proposition 3.3. in \cite{CMO4}.
\begin{lemma}
The kernel of $\Theta_F(M)$ is spanned by the elements $T_H$ where $H$ is not $F$-tight. 
\end{lemma}

\begin{definition} We will denote by $\ol{T_H}$ and $\ol{L_H}$ the images of $T_H$ and $L_H$ in $\End_{\Aut_R(M)}(\ol{KF(M)})$ respectively. 
\end{definition}
We summarize:
\begin{proposition}
The space $\End(\ol{KF(M)})$ is spanned by the elements $\{\ol{T_H}\}_{H\subseteq F\oplus F}$. All the linear relations between the elements $\ol{T_H}$ are given by:
\begin{itemize}
\item $\ol{T_H}=0$ if $H$ is not $F$-tight.
\item $\ol{L_H}=0$ if $H$ is $M$-redundant. 
\end{itemize}
\end{proposition}
\begin{proof}
The linear relations among the elements $T_H$ in $\End_{\Aut_R(M)}(KF(M))$ are given by $L_H=0$ if and only if $H$ is $M$-redundant. The kernel of $\End_{Aut_R(M)}(KF(M))\to \End_{\Aut_R(M)}(\ol{KF(M)})$ contains all the $T_H$ such that $H$ is not $F$-tight. This proves the proposition.
\end{proof}
Understanding the first type of relations in the proposition is easy. We give a set of linear generators of the space $\End_{\Aut_R(M)}(\ol{KF(M)})$, and the first type of relations says that some of them are zero. The relations of the second type are more difficult, as we need to understand how they affect the element $\ol{T_H}$ for functors $H$ that are $F$-tight. 
Let now $H\subseteq F\oplus F$ be any subfunctor. We recall Equation \ref{eq:LH}:
\begin{equation}\label{eq:LH2}L_H = T_H - \sum_i T_{H_i} + \sum_{i_1<i_2} T_{H_{i_1}\cap H_{i_2}} - \sum_{i_1<i_2<i_3}T_{H_{i_1}\cap H_{i_2}\cap H_{i_3}} + \cdots + (-1)^r T_{H_1\cap H_2\cap\cdots\cap H_r}.\end{equation}
Where $H_1,\ldots, H_r$ are the maximal proper subfunctors of $H$. The functors of the form $H_{i_1}\cap\cdots\cap H_{i_k}$ are exactly the subfunctors $H'$ of $H$ such that $H/H'$ is semisimple. We thus see that $F$-tight functors will appear in Equation \ref{eq:LH2} if and only if $H$ contains an $F$-tight subfunctor $H'$ such that $H/H'$ is semisimple. Assume that this is indeed the case. 
Since $H'$ is $F$-tight we can write $H'=H_{\alpha}$ for some $\alpha\in \Autc(F)$. We claim the following:
\begin{lemma}
The functor $H$ splits as a direct sum $H=H_{\alpha}\oplus G$, where $G\subseteq 0\oplus F$. 
\end{lemma}
\begin{proof}
The inclusion $H_{\alpha}\to H$ splits by $(1,\alpha)pr_1:H\to F\to F\oplus F$. The kernel of this map is $(0\oplus F)\cap H=G$. 
\end{proof}
\begin{lemma}\label{lem:XM}
Assume that $F$ has a minimal presentation $(X\stackrel{f}{\to} Y)$. 
If $X$ is not a direct summand of $M$ then $\ol{KF(M)}=0$. 
\end{lemma}
\begin{remark} This gives a new proof to Lemma 4.7. in \cite{CMO4}. 
\end{remark}
\begin{proof}
Assume that the conditions of the lemma hold. Consider $H_{\Id_F} = \Delta(F)\subseteq F\oplus F$. Like every $F$-tight functor $H_{\Id_F}\cong F$. Since $X$ is not a direct summand of $M$, Proposition \ref{prop:redundancycriterion} tells us that $F$ is $M$-redundant. It then follows that $\ol{L_{H_{\alpha}}}=0$ by Proposition \ref{prop:LHred}. 
But $L_{H_{\Id_F}}$ is an alternating sum of $T_{G}$'s, with $G\subseteq H_{\Id_F}$. Since the only $G\subseteq H_{\Id_F}$ which is $F$-tight and which appears in the sum is $H_{\Id_F}$ itself, and $\ol{T_G}=0$ for $G$ not $F$-tight, we get that $\ol{T_{H_{\Id_F}}}=0$. This means that the identity of the algebra $\End_{\Aut_R(M)}(\ol{KF(M)})$ is equal to zero, which can only happen when $\ol{KF(M)}=0$. 
\end{proof}
Assume from now on that $F$ has a presentation $(X\stackrel{f}{\to} Y)$, and that $X$ is a direct summand of $M$. Write $M=X\oplus Z$. 
Let $H$ be a functor that contains some $H_{\alpha}$. Then we can write $H=H_{\alpha}\oplus G$. Assume that $G$ has a minimal presentation $(X_1\stackrel{f_1}{\to} Y_1)$. 
\begin{lemma}
The functor $H$ is $M$-redundant if and only if the functor $G$ is $Z$-redundant. 
\end{lemma}
\begin{proof}
The functor $H$ splits as $H_{\alpha}\oplus G$. The minimal resolution of a direct sum of functors is the direct sum of the minimal resolutions. Since $H_{\alpha}\cong F$, $H$ has a minimal resolution of the form $(X\oplus X_1\stackrel{f\oplus f_1}{\to} Y\oplus Y_1)$. By Proposition \ref{prop:redundancycriterion} $H$ is $M$-redundant if and only if $X\oplus X_1$ is not a direct summand of $M=X\oplus Z$. Since the category $\C$ is Krull-Schimdt, this is equivalent to $X_1$ not being a direct summand of $Z$, which by Proposition \ref{prop:redundancycriterion} is equivalent to $G$ being $Z$-redundant. 
\end{proof}
To summarize, we see that the only non-zero linear relations among the elements $\ol{T_H}$ with $H$ an $F$-tight functor will arise from functors of the form $H_{\alpha}\oplus (0\oplus G)$, where $G\subseteq F$ is a $Z$-redundant semisimple functor. In the next section we will find these linear relations explicitly.

\section{concrete linear relations from $M$-redundant functors and a proof of the main theorem}\label{sec:relations}
Fix $\alpha\in \Autc(F)$ and $G\subseteq F$ such that $G$ is semisimple and $Z$-redundant. Write $H=H_{\alpha}\oplus (0\oplus G)$. We then know that $\ol{L_H}=0$, and that all linear relations among $\ol{T_H}$ for $F$-tight functors $H$ will arise from such relations. In this section we will write $\ol{L_H}$ in terms of $\ol{T_H}$ for $F$-tight functors $H$ explicitly. 

Write $H_1,\ldots H_r$ for the maximal proper subfunctors of $H$. Recall that we have 
$$L_H = T_H - \sum_i T_{H_i}+ \sum_{i_1<i_2} T_{H_1\cap H_2} - \cdots + (-1)^r T_{H_1\cap H_2\cap\cdots\cap H_r}.$$ 
We know that $\ol{T_{H'}}=0$ unless $H'$ is $F$-tight. We thus see that we can write $\ol{L_H} = \sum_{\beta\in \Autc(F)}a_{\beta} \ol{T_{H_{\beta}}}$, where
\begin{equation}\label{eq:abeta} a_{\beta} = \delta_{H_{\beta},H} -\sum_i \delta_{H_{\beta},H_i} + \cdots + (-1)^r \delta_{H_{\beta},H_1\cap\cdots\cap H_r}.\end{equation}
We thus need to understand these coefficients.

\begin{lemma}\label{lem:beta.criterion}
For $\beta\in \Autc(F)$ we have $H_{\beta}\subseteq H$ if and only if $\Im(\alpha-\beta)\subseteq G$. If this happens then $H=H_{\beta}\oplus (0\oplus G)$.
\end{lemma}
\begin{proof}
Assume first that $H_{\beta}\subseteq H$. We have $H_{\alpha}\subseteq H$. Since $H_{\alpha} = \Im(1_F,\alpha):F\to F\oplus F$ and $H_{\beta} = \Im(1_F,\beta):F\to F\oplus F$, the image of the difference of these morphisms is also contained in $H$. But the difference is just $(0,\alpha-\beta)$. Since $H\cap (0\oplus F) = 0\oplus G$, we get that $\Im(\alpha-\beta)\subseteq G$ as desired. 

The argument also works in the opposite direction: if $\Im(\alpha-\beta)\subseteq G$ then $\Im(0,\beta-\alpha)\subseteq 0\oplus G\subseteq H$ and $H_{\alpha} = \Im(\Id_F,\alpha)\subseteq H$. The sum of the morphisms is then $(1_F,\beta)$, and its image is $H_{\beta}$. The last claim is immediate, since any functor $H$ that contains an $F$-tight functor $H'$ splits as $H'\oplus (H\cap (0\oplus F))$. 
\end{proof}

Assume now that $H_{\Id_F}\subseteq H$. 
\begin{lemma}
If $H_{\beta},H_{\gamma}\subseteq H$ then $a_{\beta} = a_{\gamma}$. 
\end{lemma}
\begin{proof}
It will be enough to treat the case $\gamma = \Id_F$. 
By the last lemma it holds that $H=H_{\Id_F}\oplus (0\oplus G) = H_{\beta}\oplus (0\oplus G)$.
Also, it holds that $\Im(\Id_F-\beta)\subseteq G$, and so $G$ is stable under the action of $\beta$. 
Write $\wt{\beta}= (\Id_F,\beta)$. We have $$\wt{\beta}(H) = \wt{\beta}(H_{1_F})\oplus \wt{\beta}(0\oplus G) = \wt{\beta}(H_{1_F})\oplus (0\oplus \beta(G)) = H_{\beta}\oplus (0\oplus G)= H.$$
This means that $\wt{\beta}$ stabilizes the functor $H$. 
The automorphism $\wt{\beta}$ thus permutes the maximal proper subfunctors $H_i$ of $H$. The coefficient $a_{\Id_F}$ of $T_{H_{\Id_F}}$ in the alternating sum  
$$T_H - \sum_i T_{H_i} + \sum_{i_1<i_2} T_{H_{i_1}\cap H_{i_2}} + \cdots +(-1)^r T_{H_1\cap\cdots\cap H_r}$$ is equal to the coefficient $a_{\beta}$ of $T_{H_{\beta}}$ in the alternating sum $$T_{\wt{\beta}(H)} - \sum_i T_{\wt{\beta}(H_i)} + \sum_{i_1<i_2} T_{\wt{\beta}(H_{i_1})\cap\wt{\beta}(H_{i_2})} + \cdots + (-1)^r T_{\wt{\beta}(H_1)\cap\cdots\cap \wt{\beta}(H_r)}.$$ But this is the same sum, so we get that $a_{\beta} = a_{1_F}$ indeed. 
\end{proof}
We thus see that all the coefficients $a_{\beta}$ are equal. The more difficult part is to prove that they are also non-zero. This will be done by analysing in detail the lattice of submodules of $H/H_{\Id_F}$. 
We begin by introducing some auxiliary algebras.
\begin{definition}
Let $X$ be an object in an abelian category. Assume that $X$ has only finitely many subobjects. We write $$B_X=\bigoplus_{Y\subseteq X} K\chi_Y.$$ We define a commutative algebra structure on $B_X$ by $$\chi_{Y_1}\cdot \chi_{Y_2} = \chi_{Y_1\cap Y_2}.$$ The element $\chi_X$ is the unit of the algebra $B_X$. 
\end{definition}
We use now the expression for $L_H$ from Equation \ref{eq:LH}. We get that the coefficient $a_{\Id_F}$ is equal to the coefficient of $\chi_{H_{\Id_F}}$ in the alternating sum 
$$\chi_H-\sum_i \chi_{H_i} + \sum_{i_1<i_2}\chi_{H_1\cap H_2} - \cdots + (-1)^r\chi_{H_1\cap\cdots\cap H_r} = \prod_i (1-\chi_{H_i}),$$ where $H_1,\ldots, H_r$ are the maximal subfunctors of $H$, and the last equation takes place in the algebra $B_{H}$. 
The proof of the next lemma is straightforward.
\begin{lemma}\label{lem:BX}
\begin{enumerate}
\item The algebra $B_X$ depends only on the isomorphism type of the lattice of submodules of $X$. 
\item If $Y\subseteq X$ then we have an inclusion of algebras $$B_{X/Y}\to B_X$$ that sends $\chi_Z$ to $\chi_{p^{-1}(Z)}$, where $p:X\to X/Y$ is the canonical projection. The image contains exactly the elements $\chi_Z$ with $Y\subseteq Z$. 
\item If $G\subseteq \Aut(X)$ then $G$ acts on $B_X$ by the formula $g\cdot \chi_Y = \chi_{g(Y)}$. 
\end{enumerate}
\end{lemma}

\begin{proposition}
Assume that $H=H_{\Id_F}\oplus (0\oplus G)$, where $G\subseteq F$ is semisimple and $Z$-redundant. Then the coefficient $a_{\Id_F}$ in Equation \ref{eq:abeta} is non-zero. 
\end{proposition}\begin{proof}
By the discussion preceeding Lemma \ref{lem:BX}, we need to show that the coefficient of $\chi_{H_{\Id_F}}$ in the expression $\prod_i(1-\chi_{H_i})$ is non-zero. By the second part of Lemma \ref{lem:BX}, it is enough to show that the coefficient of $\chi_0$ in the expression $\prod_i(1-\chi_{H_i/H_{\Id_F}})$ is non-zero in $B_{H/H_{\Id_F}}\cong B_G$.  Recall that the functor $G$ is semisimple. 
Write $$G = \bigoplus_{j=1}^k S_j^{b_j},$$ where $S_j$ are non-isomorphic simple functors. Write $F_j = \End(S_j)$. Since the $S_j$'s are simple, $F_j$ is a field for every $j$. Define a commutative algebra $$A= \bigoplus_j F_je_j,$$ where the $e_j$ are mutually orthogonal idempotent. We have a natural structure of on $A$-module on $N=\bigoplus_j F_j^{b_j}$. Moreover, $\End_A(N)\cong \End_{\C_{+1}}(G)\cong \bigoplus_j \M_{b_j}(F_j)$, and there is a one-to-one order preserving correspondence between $A$-submodules of $N$ and subfunctors of $G$. Indeed, every $A$-submodule of $N$ is of the form $eN$ for some idempotent $e\in \End_A(N)$. The corresponding subfunctor is given by $eG$, where we freely use the identification $\End_A(N)\cong \End_{\C_{+1}}(G)$. 
By the first part of Lemma \ref{lem:BX}, it will be enough to prove that the coefficient of $\chi_0$ in the expression $\prod_i(1-\chi_{N_i})$ is non-zero, where $\{N_i\}$ is the set of all maximal propers $A$-submodules of $N$. 

The orthogonal idempotents $e_j$ in $A$ give rise to a tensor product decomposition \begin{equation}\label{eq:tensorproduct} B_N \cong \bigotimes_j B_{e_jN}.\end{equation} The isomorphism $\bigotimes_j B_{e_jN}\to B_N$ is given by the product of the maps $B_{e_jN}\to B_N$ described in Part (2) of Lemma \ref{lem:BX} followed by multiplication. If $N'\subseteq N$ is a maximal submodule, then $N/N'$ must be simple, and is therefore isomorphic to $S_j$ for some $j$. It follows that $N'$ contains $e_{j'}N$ for all $j'\neq j$, and is in fact the inverse image of a maximal submodule of $e_jN$ under the projection $N\to e_jN$. 
It then holds that under the isomorphism \ref{eq:tensorproduct} the product $\prod_i(1-\chi_{N_i})$ corresponds to $\prod_j\prod_i(1-\chi_{N_{ji}}),$ where $N_{ji}$ are the maximal submodules of $e_jN$. It will thus be enough to prove that the coefficient of $\chi_0$ is non-zero in the product $\prod_i(1-\chi_{N_{ji}})$ for a specific $j$. In other words, we can assume that the number of isomorphism types of direct summands that appear in $G$ is 1. This means that the algebra $A$ is just the field $\F_q$, and $N$ is just $\F_q^b$ for some $b\in\N$. Submodules are now just $\F_q$-vector subspaces of $\F_q^b$.

So we reduce to the following case: we consider the element $$P_b=\prod_{\substack{V\subseteq \F_q^b \\ \dim(V) = b-1}}(1-\chi_V)$$ in the algebra $B_{\F_q^b}$. We would like to prove that the coefficient of $\chi_0$ in $P_b$ is non-zero. Write $P_b = \sum_{V\subseteq \F_q^b}a_V \chi_V$. 
We first claim that $a_V$ only depends on the dimension of $V$. This follows easily from the fact that the group $\GL_b(\F_q)$ acts transitively on vector subspaces of $\F_q^b$ of a given dimension and stabilizes $P_b$. 
So we can write 
\begin{equation}\label{eq:Pb}P_b = \sum_{m=0}^b a_{m,b}\sum_{\substack{V\subseteq \F_q^b \\ \dim(V) = m}}\chi_V.\end{equation}

Next, we claim that $a_{m,b} = a_{0,b-m}$. To see this, we fix a vector space $W\subseteq \F_q^b$ of dimension $m$. We then write 
$$P_b = \prod_{V\in X_1}(1-\chi_V)\cdot \prod_{V\in X_2}(1-\chi_V),$$ where $X_1 = \{V| W\subseteq V\subseteq \F_q^b, \dim(V) = b-1\}$ and $X_2 =  \{V| W\not \subseteq V\subseteq \F_q^b, \dim(V) = b-1\}$. The coefficient of $\chi_W$ in $P$ will then be the coefficient of $\chi_W$ in $\prod_{V\in X_1}(1-\chi_V)$. But this is exactly the image of $P_{m-b}$ in $B_{\F_q^b}$ under the map $B_{\F_q^{m-b}} = B_{\F_q^m/W}\to B_{\F_q^m}$. This implies that $a_{m,b} = a_{0,m-b}$ indeed. We will write $a_m = a_{0,m}$ henceforth. 

By considering the cases $m=0$ and $m=1$ we find that $a_0=1$ and $a_1=-1$. Recall that the number of $m$-dimensional subspaces of $\F_q^b$ is given by the $q$-binomial coefficient
$$\frac{(q^b-1)(q^b-q)\cdots (q^b-q^{m-1})}{(q^m-1)(q^m-q)\cdots (q^m-q^{m-1})} = \binom{b}{m}_q.$$
Apply the ring homomorphism $B_{\F_q^b}\to K$ that sends all $\chi_V$ to 1 for every $V\subseteq \F_q^b$ to Equation \ref{eq:Pb}. We get
$$\sum_{m=0}^b a_m \binom{b}{b-m}_q = \sum_{m=0}^b a_m \binom{b}{m}_q = 0,$$ which gives a linear recursion for the numbers $a_m$ (we have used here the fact that $\binom{b}{m}_q = \binom{b}{b-m}_q$). To actually find the $a_m$'s we use now the $q$-binomial formula, which states that 
$$\prod_{k=0}^{b-1}(1+q^kt) = \sum_{k=0}^b q^{k(k-1)/2}\binom{b}{k}_qt^k.$$ The substitution $t=-1$ gives 
$$0= \sum_{k=0}^b q^{k(k-1)/2}\binom{b}{k}_q(-1)^k.$$ This shows that $a_m = q^{m(m-1)/2}(-1)^m$ satisfies the above linear recursion. It follows that $a_m\neq 0$ as desired.
\end{proof}
We now deduce some corollaries. 
\begin{definition} For a subfunctor $G\subseteq F$ we define the subgroup $$P_G = \{\phi\in \Autc(F) | \Im(\phi-\Id_F)\subseteq G\}.$$
\end{definition}
\begin{remark} If $F=G\oplus G'$ then we can think of $P_G$ as the group of all automorphisms that have the block diagonal form 
$$\begin{pmatrix} * & * \\ 0 & I \end{pmatrix}.$$ This is the reason we chose the terminology $P_G$, as it is related to parabolic subgroups.
\end{remark} 
Now, we have see that if $H=H_{\Id_F}\oplus (0\oplus G)$ is an $M$-redundant functor with $G$ semisimple, then up to the non-zero scalar $a_{\Id_F}$ it holds that 
$$\ol{L_H} = \sum_{\substack{\beta\in \Autc(F) \\ H_{\beta}\subseteq H}}\ol{T_{H_{\beta}}}.$$ By Lemma \ref{lem:beta.criterion} we have 
$$\ol{L_H} = \sum_{\beta\in P_G}\ol{T_{H_{\beta}}}.$$ 
If $H = H_{\alpha}\oplus (0,G)$ is an $M$-redundant functor with $G$-semisimple and $Z$-redundant functor, and $\alpha\in \Autc(F)$, then we can apply the automorphism $(\alpha,1)$ that sends $H_{\alpha}$ to $H_{\Id_F}$ to reduce to the case where $\alpha=\Id_F$. Indeed, in this case we will get that
$$\ol{L_H} = \sum_{\beta\in P_G}\ol{T_{H_{\beta}}T_{H_{\alpha}}}.$$ 
Notice that it holds that for $g\in \Autc(F)$ we have $gP_Gg^{-1} = P_{g(G)}$. 
We summarize all of this in the following:
\begin{proposition}
Assume that $M = X\oplus Z$ and that $F$ has a minimal presentation $(X\stackrel{f}{\to}Y)$. Consider the homomorphism $$\Phi:K\Autc(F)\to \End_{\Aut_R(M)}(\ol{KF(M)})$$ $$\alpha\mapsto T_{H_{\alpha}}.$$ Then the kernel of $\Phi$ is the two-sided ideal generated by the idempotents $e_{P_G}$, where $G\subseteq F$ is a semisimple $Z$-redundant subfunctor.
\end{proposition}
Finally, we can state the main theorem:
\begin{theorem}\label{thm:main}
 Assume that $M = X\oplus Z$ and that $F$ has a minimal presentation $(X\stackrel{f}{\to}Y)$. An irreducible representation $V$ of $\Autc(F)$ is the functor morphing of an irreducible representation of $\Aut_R(M)$ if and only if the following condition holds: whenever $G\subseteq F$ is a semisimple subfunctor with a projective cover $\Hom(Z_G,-)$ such that $Z_G$ is not a direct summand of $Z$ we have $\Hom_{P_G}(\one,V)=0$. 
\end{theorem}
\begin{proof}
In a semisimple algebra $\bigoplus\End(V_i)$, the two-sided ideal generated by an idempotent $e$ contains exactly the direct sum of those $\End(V_i)$ such that $e\cdot V_i\neq 0$. But for an irreducible representation $V$ of $\Autc(F)$, $e_{P_G}\cdot V=0$ if and only if $\Hom_{P_G}(\one,V)=0$.
\end{proof}

\begin{corollary}
Let $M=X\oplus Z$ and let $(X\xrightarrow{f} Y)$ be a minimal resolution for $F$. Then the following conditions are equivalent:
\begin{enumerate}
\item Every irreducible representation of $\Autc(F)$ is the functor morphing of an irreducible representation of $\Aut_R(M)$. 
\item The trivial representation $\one$ of $\Autc(F)$ is the functor morphing of some irreducible representation of $\Aut_R(M)$.
\item If $\Hom_R(X_1,-)$ is the projective cover of $soc(F)$ then $X_1$ is a direct summand of $Z$. 
\end{enumerate}
\end{corollary}
\begin{proof}
It is clear that 1 implies 2. If there is a semisimple functor $G\subseteq F$ with a projective cover $\Hom_R(Z_G,-)$ such that $Z_G$ is not a direct summand of $Z$, then the theorem shows that $\one$ is not the functor morphing of an irreducible representation of $\Aut_R(M)$. So if $\one$ is the functor morphing of a representation of $\Aut_R(M)$, this cannot happen, and so every module $X_2$ that appears in the projective cover of any semisimple subfunctor $G\subseteq F$ must be a direct summand of $Z$. This is true in particular for $G=soc(F)$, which shows that 2 implies 3.
Finally, if 3 holds then the conditions of the theorem hold trivially, and so every irreducible representation of $\Autc(F)$ is the functor morphing of an irreducible representation of $\Aut_R(M)$. 
\end{proof}
We record the following lemma, which will be useful to analyse some of the examples in the next section.
\begin{lemma}\label{lem:functor.inclusion}
If $G_1\subseteq G_2$ are semisimple functors then $P_{G_1}\subseteq P_{G_2}$. As a result, when applying Theorem \ref{thm:main} we only need to consider the minimal functors (with respect to inclusion) in $\{G| Z_G\text{ is not a direct summand of } Z\}$.\end{lemma}
\begin{proof}
The subgroup $P_{G_1}$ contains all automorphisms of the form $\phi=1+\psi$, where $\psi:F\to G_1\to F$. Since $G_1\subseteq G_2$, such automorphisms are clearly contained in $P_{G_2}$ as well. 
This implies that if $\Hom_{P_{G_1}}(\one,V)=0$ then $\Hom_{P_{G_2}}(\one,V)=0$, and the result follows.
\end{proof}

\section{Examples}\label{sec:examples}
\subsection{The groups $\GL_n(\F_q)$}
Assume that $R=\F_q$ is the finite field with $q$ elements, and $M = \F_q^n$. It holds that $\Aut_R(M) = \GL_n(\F_q)$. 
The functors in $\C_{+1}$ are all of the form $F=\Hom_{\F_q}(\F_q^m,-)$, and we have an equivalence $\C_{+1}\cong \C = \Vect_{\F_q}$. For $F=\Hom_{\F_q}(\F_q^m,-)$ it holds that $X=\F_q^m$ and $X$ is a direct summand of $M$ if and only if $m\leq n$. In this case we write $M = X\oplus Z$ where $Z = \F_q^{n-m}$. 

Next, we notice that all subobjects of $F$ are semisimple. A subobject $G$ of $F$ is isomorphic to $\Hom_{\F_q}(\F_q^k,-)$ for some $k\leq m$, and if $G_1\cong G_2\subseteq F$ are two isomorphic subfunctors then they are conjugate under the action of $\Autc(F)\cong \GL_m(\F_q)$. 
It holds that $\F_q^k$ is not a direct summand of $Z$ if and only if $k>n-m$. Since $k\leq m$ this inequality is possible only when $n-m<m$ , or $m>n/2$. 
We thus see that if $m\leq n/2$ then every representation of $\GL_m(\F_q)$ comes from functor morphing of a representation of $\GL_n(\F_q)$. If $m>n/2$ then 
by Lemma \ref{lem:functor.inclusion} it is enough to consider the functor $G=\Hom_{\F_q}(\F_q^{n-m+1},-)$, as $\F_q^{n-m+1}$ is the minimal $\F_q$-vector space that is not a direct summand of $Z$. Write $k=n-m+1$ henceforth.

We have $$P_G = \{\begin{pmatrix} A & B \\ 0 & I \end{pmatrix} | A\in \M_{k\times k}(\F_q)\, B\in \M_{k,m-k}(\F_q)\}\subseteq \Autc(F).$$
To understand what representations $V$ of $\Autc(F)$ satisfy $\Hom_{P_G}(\one, V)=0$, we recall Zelevinsky's classification of irreducible representations of $\GL_n(\F_q)$ using PSH-algebras from \cite{Zelevinsky}. 

Write $$\H = \bigoplus_{a\geq 0} K_0(\GL_a(\F_q)).$$ Then $\H$ has a structure of a positive self-adjoint Hopf algebra (or PSH-algebra), where multiplication is graded and given by by Harish-Chandra (parabolic) induction $$K_0(\GL_a(\F_q))\ot K_0(\GL_b(\F_q))\cong K_0(\GL_a(\F_q)\times \GL_b(\F_q))\stackrel{inf}{\to} K_0(P_{a,b})\stackrel{Ind}{\to}K_0(\GL_{a+b}(\F_q)),$$
and comultiplication is the dual operator, given by Harish-Chandra (parabolic) restriction
$$K_0(\GL_{a+b}(\F_q))\stackrel{res}{\to} K_0(P_{a,b})\stackrel{inv_{U_{a,b}}}{\to}K_0(\GL_a(\F_q)\times \GL_b(\F_q))\cong K_0(\GL_a(\F_q))\ot K_0(\GL_b(\F_q)).$$
Here $$P_{a,b} = \{\begin{pmatrix} A & B \\ 0 & C \end{pmatrix} | A\in \M_{a\times a}(\F_q), B\in \M_{a\times b}(\F_q), C\in \M_{b\times b}(\F_q)\}$$ and 
$$U_{a,b} = \{\begin{pmatrix} I & B \\ 0 & I \end{pmatrix}\}\subseteq P_{a,b}$$ is the kernel of the natural projection $$P_{a,b}\to \GL_a(\F_q)\times \GL_b(\F_q)$$ $$\begin{pmatrix} A & B \\ 0 & C \end{pmatrix}\mapsto (A,C).$$ Notice that the group $P_G$ is the subgroup of all matrices in $P_{k,m-k}$ where $C=I$. 

The algebra $\H$ has a natural $\Z$-basis given by irreducible representations. The term positive self adjoint refers to the fact that all the structure constants with respect to this basis are non-negative integers, and that the multiplication is dual to the comultiplication.
Zelevinsky showed that there is a unique (up to rescaling of the grading) \textit{universal} PSH-algebra, and that every PSH-algebra splits as a tensor product of universal PSH-algebras in a unique way. The universal PSH-algebra is $Zel=\bigoplus_{a\geq 0} K_0(S_a)$, where the multiplication and comultiplication come from induction and restriction respectively. Since the irreducible representations of $S_a$ are paramterized by partitions of $a$, the universal PSH-algebra has a basis $w_{\la}$ where $\la\vdash a$ for some $a$.
Zelevinsky showed that $Zel\cong \Z[x_1,x_2,\ldots]$. Here $x_a=w_{(a)}$ where $(a)\vdash a$ is the partition $a=a$. The element $x_a$ corresponds to the trivial representation of $S_a$. The comultiplication for the elements $x_1$ is given by $\Delta(x_c) = \sum_{a+b=c} x_a\ot x_b$.  

Going back to the algebra $\H$, we can write $$\H = \bigotimes_{\rho}\H_{\rho},$$ where the tensor product is taken over all cuspidal elements in $\GL_a(\F_q)$ for some $a>0$. 
This means in particular that for every irreducible representation $V$ of $\GL_a(\F_q)$ we have a unique decomposition 
$$[V] = \prod_{\rho} w_{\rho,\la_{\rho}},$$ where for every cuspidal $\rho$, $\la_{\rho}\vdash a_{\rho}$ is a partition of some natural number, and  $w_{\rho,\la_{\rho}}$ is the irreducible representation in $\H_{\rho}$ that corresponds to $w_{\la_{\rho}}$ under the isomorphism $H_{\rho}\cong Zel$. The numbers $a_{\rho}$ are zero for almost all $\rho$ 

Going back to the group $P_G$, the statement $\Hom_{P_G}(\one, V)\neq 0$ is equivalent to $V\in \Rep(\GL_m(\F_q))$ being a direct summand of $ \one_k\cdot W$ for some $W\in \Rep(\GL_{m-k}(\F_q))$, where $\one_k$ stands for the trivial representation of $\GL_k(\F_q)$. Write $x_k = [\one_k]$. Then the above statement is equivalent to $\langle x_k^*([V]),[W]\rangle=\langle [V],x_k\cdot [W]\rangle \neq 0,$ where $x_k^*$ is the adjoint operator to multiplication by $x_k$. Since $W$ is arbitrary, this is equivalent to $x_k^*([V])\neq 0$. Now $x_k$ belongs to $\H_\epsilon$, where $\epsilon:\GL_1(\F_q)\to K^{\times}$ is the trivial representation. It then holds that if we write $[V] = \prod_{\rho} w_{\rho,\la_{\rho}}$ then $$x_k^*([V]) = x_k^*(w_{\epsilon,\la_{\epsilon}})\prod_{\rho\neq \epsilon} w_{\rho,\la_{\rho}}.$$

Summarizing, we see that $\Hom_{P_G}(\one,V)=0$ if and only if $x_k^*(w_{\epsilon,\la_{\epsilon}})=0$. By the formula on page 50 of \cite{Zelevinsky}, this happens if and only if $(\la_{\epsilon})_1<k$, where $(\la_{\epsilon})_1$ stands for the length of the first row in the partition  $\la$ (the formula in \cite{Zelevinsky} is for $X^*=\sum_{k\geq 0} x_k^*$, but by grading considerations we get a formula for $x_k^*$ as well). 
We summarize this discussion.
\begin{proposition}
An irreducible representation $V$ of $\GL_m(\F_q)$ comes from functor morphing of an irreducible representation of $\GL_n(\F_q)$ if and only if the following condition holds:
in the decomposition $[V] = \prod_{\rho}w_{\rho,\la_{\rho}}$ it holds that $(\la_{\epsilon})_1< n-m+1$. 
\end{proposition}
In \cite{CMO4} we have described explicitly the functor morphing of an irreducible representation of $\GL_n(\F_q)$. We have shown that if $[W] = \prod_{\rho}w_{\rho,\mu_{\rho}}$ then the functor morphing of $W$ is a representation $V$ of $\GL_m(\F_q)$, where $m = n - (\mu_{\epsilon})_1$ that is given by 
$$[V] = w_{\epsilon,\la_{\epsilon}}\prod_{\rho\neq \epsilon} w_{\rho,\mu_{\rho}},$$ where $\la_{\epsilon} = ((\mu_{\epsilon})_2,(\mu_{\epsilon})_3,\cdots)$. 
it is easy to deduce the above proposition from this description of functor morphing. 

\subsection{The groups $B_n(\F_q)$}

Let $n\geq 1$ and let $\F_q$ be the finite field with $q$ elements. 
Consider the ring $R = span_{\F_q}\{e_{ij}\}_{i\geq j}\subseteq \M_n(\F_q)$ of lower triangular matrices over $\F_q$. 
We can think of $R$ as the quiver algebra of the quiver $A_n$. 
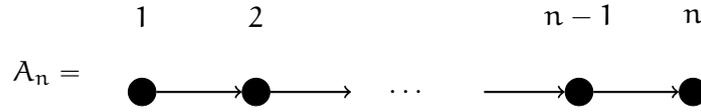
\begin{center}\scalebox{1.0}{
\begin{tikzpicture}
	\begin{pgfonlayer}{nodelayer}
		\node [style=graph node] (0) at (-6.25, 4.25) {};
		\node [style=graph node] (1) at (-4.75, 4.25) {};
		\node [style=graph node] (2) at (-0.5, 4.25) {};
		\node [style=graph node] (3) at (1, 4.25) {};
		\node [style=none] (4) at (-6.25, 5.25) {$1$};
		\node [style=none] (5) at (-4.75, 5.25) {$2$};
		\node [style=none] (6) at (-0.5, 5.25) {$n-1$};
		\node [style=none] (7) at (1, 5.25) {$n$};
		\node [style=none] (8) at (-3.5, 4.25) {};
		\node [style=none] (9) at (-1.75, 4.25) {};
		\node [style=none] (10) at (-2.75, 4.25) {$\cdots$};
		\node [style=none] (11) at (-7.5, 4.5) {$A_n=$};
	\end{pgfonlayer}
	\begin{pgfonlayer}{edgelayer}
		\draw [style=new edge style 0] (0) to (1);
		\draw [style=new edge style 0] (1) to (8.center);
		\draw [style=new edge style 0] (9.center) to (2);
		\draw [style=new edge style 0] (2) to (3);
	\end{pgfonlayer}
\end{tikzpicture}
}
\end{center}

For $i=1,\ldots,n$ we write $M_j = Re_{jj}= span_{\F_q}\{e_{ij}\}_{i\geq j}$. The hom space $\Hom_R(M_i,M_j)$ is zero if $i<j$, and is one dimensional and spanned by 
$$f_{ij}:M_i\to M_j $$
$$e_{ki}\mapsto e_{kj}$$ if $i\geq j$. If $i\leq j\leq k$ then it holds that $f_{jk}f_{ij} = f_{ik}$. Consider the module $M=\bigoplus_{i=1}^n M_n$. It holds that $R=M$ as an $R$-module, and we can naturally identify $\End_R(M)\cong R^{op}$ with the ring of upper triangular matrices over $\F_q$. The group $\Aut_R(M)$ is the group $B_n(\F_q)$ of invertible upper triangular matrices over $\F_q$.  
Write $\C = \langle M_1,\ldots, M_n\rangle$. Every functor $F\in\C_{+1}=Fun(\C,Ab)$ is determined by its values on $R$, and the category $\C_{+1}$ is equivalent to the category of all finite $R^{op}$-modules. From the classification of representations of the quiver $A_n$ we know that isomorphism classes of indecomposable functors correspond to intervals $[i,j]$ with $i\leq j$. The functor $F_{ij}$ that corresponds to $[i,j]$ is given by 
$$F_{ij}(M_k) = \begin{cases} \F_q & \text{ if } i\leq k\leq j \\ 0 & \text{ otherwise}\end{cases}$$ on objects, and on morphisms we have 
$$F_{ij}(f_{kl}) = \begin{cases} \Id_{\F_q} & \text{ if } i\leq l\leq k\leq j\\ 0 & \text{ otherwise}\end{cases}.$$
The functor $F_{ij}$ has the minimal resolution 
$$\Hom_R(M_{i-1},0)\xrightarrow{f_{j,i-1}^*}Hom_R(M_j,-)\rightarrow F_{ij}\to 0,$$ where we write $M_0=0$ and $f_{j0}=0$ for convenience. 
The simple functors are then $F_{ii}$, and the socle of $F_{ij}$ is $F_{ii}$. Moreover, we can describe explicitly the hom-spaces between the indecomposable functors. We have
$$\Hom_{\C_{+1}}(F_{ij},F_{kl}) = \begin{cases} \F_q & \text{ if } i\leq k\leq j\leq l \\ 0& \text{ otherwise}\end{cases}.$$ If $i\leq k\leq j\leq l$, $k\leq r\leq l\leq s$ and $i\leq r\leq j\leq s$ then composition of morphisms $$\Hom_{\C_{+1}}(F_{kl},F_{rs})\times \Hom_{\C_{+1}}(F_{ij},F_{kl})\to \Hom_{\C_{+1}}(F_{ij},F_{rs})$$ is given by multiplication $\F_q\times \F_q\to \F_q$. In all other cases it is zero. 

With this terminology in hand we study the functor morphing of the irreducible representations of $\Aut_R(M)\cong B_n(\F_q)$.
Let $F=\bigoplus F_{ij}^{x_{ij}}$ for some integers $x_{ij}$. The projective cover of $F$ is $\Hom_R(\bigoplus_j M_j^{y_j},-)$, where $y_j = \sum_i x_{ij}$. By Lemma \ref{lem:XM}. we know that if $F$ arises in functor morphing of representations of $\Aut_R(M)$ then $\bigoplus_j M_j^{y_j}$ is a direct summand of $M=\bigoplus M_j$. We thus get the following:
\begin{lemma}
If $F$ arises in functor morphing from $\Aut_R(M)$, then $x_{ij}\leq 1$ for all $i,j$, and for any given $j$ there is at most one index $i$ such that $x_{ij}=1$. 
\end{lemma}
In other words, we can write $F = \bigoplus_{i=1}^l F_{a_i,b_i}$, where $b_i\neq b_j$ for $i\neq j$. 
It then holds that the projective cover of $F$ is given by $\Hom_R(\bigoplus_{i=1}^l M_{b_i},-)$.
We can write $M = \bigoplus_{i=1}^l M_{b_i}\oplus Z$, where $Z=\bigoplus_{i\in S} M_i $ for $S = [n]\backslash \{b_i\}$.  

The criterion for the $b_i$'s also works for the $a_i$'s.
\begin{lemma} If $F$ arises from functor morphing of a representation of $\Aut_R(M)$ then $a_i\neq a_j$ for $i\neq j$. 
\end{lemma}
\begin{proof} Assume that there are $i\neq j$ such that $a_i=a_j$. Write $a=a_i=a_j$. We have $G:=soc(F_{a_i,b_i}\oplus F_{a_j,b_j}) = F_{aa}\oplus F_{aa}$. The projective cover of $G$ is then $\Hom(M_a^2,-)$, and $M_a^2$ is clearly not a direct summand of $Z$. We can thus check the criterion of the main Theorem. 

If there is no index $k$ such that $b_k=a$ then $\Hom_{\C_{+1}}(F,G)=0$. Since $P_G = \Autc(F)\cap (\Id_F + \Hom(F,G))$, it holds that $P_G= \{\Id_F\}$. But then every irreducible representation $V$ of $\Autc(F)$ satisfies $\Hom_{P_G}(\one,V)\neq 0$, and thus $V$ does not arise as functor morphing of a representation in $\Aut_R(M)$.

Assume then that there is an index $k$ such that $b_k=a$. In this case it holds that $M_a$ is a direct summand of $X$ and therefore not a direct summand of $Z$.
We will distinguish between two cases. Assume first that $b_i\neq a\neq b_j$.
It holds that $\End_{\C_{+1}}(F_{aa}) = \F_q$. For $[x:y]\in \P^1(\F_q)$ we write $G_{[x:y]}$ for the image of the map $F_{aa}\xrightarrow{(x,y)} F_{aa}\oplus F_{aa}=G$. It can easily be seen that this does not depend on the representative of the element in $\P^1(\F_q)$, and that $G_{[x:y]}$ is isomorphic to $F_{aa}$. 
The group $P_G$ contains all automorphisms of the form $\phi=1+\psi$, where $\psi:F\twoheadrightarrow F_{a_k,a}\to G\to F$. The first morphism here is the projection of $F$ onto its direct summand $F_{a_k,a}$, the second morphism is given by some $(x,y)\in \F_q^2$, and the last morphism is the inclusion of $G$ in $F$. We will write such a morphism as $\psi=\psi_{(x,y)}$. Since $b_i\neq a\neq b_j$ it holds that for every such $\psi_{(x,y)}$ the resulting $\phi$ is invertible. Moreover, $\psi_{(x,y)} + \psi_{(z,w)} = \psi_{(x+z,yy+w)}$. We then have $(1+\psi_{(x,y)})(1+\psi_{(z,w)}) = 1+\psi_{(x,y)}+ \psi_{(z,w)} = 1+\psi_{(x+z,y+w)}$, and the group $P_G$ is thus isomorphic to $\F_q^2$. 

Choose now a non-trivial additive character $\zeta:\F_q\to K^{\times}$. Every character on $\F_q$ is then of the form $a\mapsto \zeta(xa)$ for some $x\in \F_q$. 
In this way, the character group of $P_G$ is also isomorphic to $\F_q^2$. We will denote by $\zeta_{(x,y)}:P_G\to K^{\times}$ the one-dimensional character that is given by $(z,w)\mapsto \zeta(zx+yw)$, where we identify $P_G$ with $\F_q^2$. 

Let now $V$ be an irreducible representation of $\Autc(F)$. Consider the restriction of $V$ to $P_G$. Decomposing $V$ to irreducible $P_G$-representations, we can find a non-zero vector $v\in V$ such that $g\cdot v = \zeta_{(x,y)}(g)v$ for some $(x,y)\in \F_q^2$. If $(x,y)=(0,0)$ then $\Hom_{P_G}(\one,V)\neq 0$, and this shows that $V$ is not the functor morphing of any representation of $\Aut_R(M)$. If $(x,y)\neq (0,0)$ then a direct calculation shows that $g\cdot v = v$ for $g\in G_{[y:-x]}$. It follows that $\Hom_{P_{G_{[y:-x]}}}(\one,V)\neq 0$, and again $V$ is not the functor morphing of any representation of $\Aut_R(M)$. 

The second case that we need to consider is when $b_i=a$ or $b_j=a$. Since $b_i\neq b_j$ we can assume without loss of generality that $b_i=a$. Since $b_k=a$ as well, it follows that $k=i$. The subgroup $P_G$ contains now all the automorphisms of the form $1+\psi$, where $\psi:F\twoheadrightarrow F_{a,a}\to F_{a,a}\oplus F_{a,a}=G\to F$. Here the first copy of $F_{a,a}$ is the direct summand $F_{a,a}$ of $F$, which we denote by $G_1$. The second copy of $F_{a,a}$ is the socle of $F_{a,b_j}$ which we denote by $G_2$. 
The morphism $\psi$ is given, as before, by $(x,y)\in \F_q^2$, However, in this case the morphism $\phi = 1+\psi$ is invertible if and only if $x\neq -1$. The resulting group $P_G$ is isomorphic to $\F_q^{\times}\ltimes \F_q$.  The group $P_{G_1}$ is then the subgroup $\F_q^{\times}$ of $P_G$, and the group $P_{G_2}$ is $\F_q$. 

Let now $V$ be an irreducible representation of $\Autc(F)$. We consider its restriction to $P_G= \F_q^{\times}\ltimes \F_q$. We would like to prove that $\Hom_{P_{G_1}}(\one,V)\neq 0$ or that $\Hom_{P_{G_2}}(\one,V)\neq 0$. This will show that $V$ cannot arise as the functor morphing of a representation of $\Aut_R(M)$. By decomposing $V$ into irreducible $P_G$ representations, it will be enough to prove the claim for an irreducible $P_G$-representation 

So assume that  $\Hom_{P_{G_2}}(\one,V)=0$. Since $\Hom_{P_{G_2}}(\one,V)=0$, it must hold that $\F_q$ acts on $V$ by non-trivial characters.  The group $\F_q^{\times}$ acts freely on the set of non-trivial characters of $\F_q$, so by Clifford Theory we get $V = \Ind_{\F_q}^{\F_q^{\times}\ltimes \F_q} \wt{V}$ for some representation $\wt{V}$ of $\F_q$. But this means that $V$ is a regular $\F_q^{\times}$-representation, which means that in particular $\Hom_{P_{G_1}}(\one,V)\neq 0$. 

This covers all the cases in which $a_i=a_j$, and it shows us that when it happens no irreducible representation of $\Autc(F)$ appears as the functor morphing of any irreducible representation of $\Aut_R(M)$.
\end{proof}

The last two lemmas tell us that only functors of the form $F=\bigoplus_{i=1}^l F_{a_i,b_i}$ where $a_i\neq a_j$ and $b_i\neq b_j$ when $i\neq j$ can arise in functor morphing from representations of $\Aut_R(M)$. The projective cover of such a functor $F$ is thus $\Hom_R(\bigoplus_{i=1}^l M_{b_i},-)$, and the socle of $F$ is $$soc(F) = \bigoplus_{i=1}^l F_{a_i,a_i}.$$ We thus have $X=\bigoplus_{i=1}^l M_{b_i}$ and $Z= \bigoplus_{s\in S} M_s$ where $S=[n]\backslash \{b_i\}$.
The semisimple functors $G$ that we will have to consider are thus all of the form $\bigoplus_{i\in T}F_{a_i,a_i}$, where $T\subseteq [l]$ satisfies that for some $i\in T$ it holds that $a_i\in \{b_j\}$. 
By Lemma \ref{lem:functor.inclusion} we only need to consider functors of the form $G=F_{a_i,a_i}$ where $a_i=b_j$ for some $j$.
\begin{lemma}
If $G=F_{a_i,a_i}$ and $a_i=b_j$ then $P_G$ is a normal subgroup of $\Autc(F)$. If $i\neq j$ then $P_G\cong \F_q$ and if $i=j$ then $P_G\cong \F_q^{\times}$. 
\end{lemma}
\begin{proof}
We write $F=\bigoplus_k F_{a_k,b_k}$. The only $k$ such that $\Hom_{\C_{+1}}(F_{a_k,b_k},F_{a_i,a_i})\neq 0$ is $j$. The only $k$ such that $\Hom_{\C_{+1}}(F_{a_i,a_i},F_{a_k,b_k})\neq 0$ is $i$. 
We distinguish now between two cases. 

If $i= j$ then $a_i=b_i$. There are no non-zero morphisms $F_{a_i,b_i}\to F_{a_k,b_k}$ and $F_{a_k,b_k}\to F_{a_i,a_i}$ when $k\neq i$. The ring $\End_{\C_{+1}}(F)$ splits as $\End_{\C_{+1}}(F_{a_i,a_i})\oplus \End_{\C_{+1}}(\bigoplus_{k\neq i} F_{a_k,b_k})$, and as a result $$\Autc(F)=\Autc(F_{a_i,a_i})\times \Autc(\bigoplus_{k\neq i} F_{a_k,b_k}).$$ Since $G=F_{a_i,a_i}$ and $\End_{\C_{+1}}(F_{a_i,a_i})=\F_q$, we get that $P_G = \F_q^{\times}$ is a direct summand of $\Autc(F)$, and in particular normal.

If $i\neq j$ then $I = \Hom_{\C_{+1}}(F_{a_j,b_j},F_{a_i,b_i})\subseteq \Hom_{\C_{+1}}(\bigoplus_k F_{a_k,b_k}, F_{a_k,b_k})=\End_{\C_{+1}}(F)$ is a nilpotent two-sided ideal. 
This follows from the fact that for any $k\neq i$ any composition $F_{a_j,b_j}\to F_{a_i,b_i}\to F_{a_k,b_k}$ vanishes, because $a_i\neq a_k$ and so  $b_j=a_i< a_k$, and for any $k\neq j$ the composition 
 $F_{a_k,b_k}\to F_{a_j,b_j}\to F_{a_i,b_i}$ vanishes due to a similar reason. The group $P_G$ is then isomorphic to $1+I$, which is isomorphic to $\F_q$. 
\end{proof}
We summarize this discussion
\begin{proposition}
Let $F:=\bigoplus_i F_{a_i,b_i}$. Then $F$ arises in functor morphing of representations of $\Aut_R(M)$ if and only if $a_i\neq a_j$ and $b_i\neq b_j$ for $i\neq j$. When this happens, an irreducible representation $V$ of $\Autc(F)$ comes from functor morphing of a representation of $\Aut_R(M)$ if and only if the following condition holds: if $a_i=b_i$ then the subgroup $\F_q^{\times}\cong \Autc(F_{a_i,a_i})<\Autc(F)$ acts non-trivially on $V$, and if $a_i=b_j$ for $i\neq j$ then the subgroup $\F_q\cong 1+\Hom_{\C_{+1}}(F_{a_j,b_j},F_{a_i,b_i})$ acts non-trivially on $V$. 
\end{proposition}
\begin{proof} 
The only new thing in the proposition is the fact that if indeed $a_i\neq a_j$ and $b_i\neq b_j$ for all $i\neq j$ then such a representation exists. But this is straightforward, by considering the product of all the $P_G$ subgroups, which is again abelian in this case. 
\end{proof}
We can get rid of the singletons among the intervals $[a_i,b_i]$ without too much hassle. 
For a given functor $F=\bigoplus_{i=1}^m F_{a_i,b_i}$ we can assume, by reordering the intervals, that $[a_1,b_1],\ldots, [a_k,b_k]$ are not singletons and $[a_{k+1},b_{k+1}],\ldots,[a_m,b_m]$ are singletons (i.e. $a_{k+1}=b_{k+1},\ldots, a_m=b_m$). Write $F' = \bigoplus_{i=1}^k F_{a_i,b_i}$. Then $\Autc(F)\cong \Autc(F')\times (\F_q^{\times})^{m-k}$. Every irreducible representation $V$ of $F$ can be written uniquely as $V=V'\ot K_{\psi_1,\ldots, \psi_{m-k}}$, where $\psi_j:\F_q\to K$ is a one-dimensional character and $V'\in \Irr(\Autc(F'))$.  
Recall that $\Irr^{M,fm}(\Autc(F))$ are the irreducible representations of $\Autc(F)$ that arise from functor morphing of irreducible representations of $\Aut_R(M)$.
The above proposition gives the following:
\begin{corollary}
It holds that $V\in\Irr^{M,fm}(\Autc(F))$ if and only if $V'\in \Irr^{M,fm}(\Autc(\Irr(F'))$ and all the characters $\psi_j$ are non-trivial. 
\end{corollary}

\begin{corollary} Let $F'=\bigoplus_{i=1}^k F_{a_i,b_i}$ where $a_i\neq b_i$ for all $i$. Write $l = |[n]\backslash\{a_1,b_1,\ldots, a_k,b_k\}|$.
Then there is a one-to-one correspondence $$\bigsqcup_F \Irr^{M,fm}(F)\cong \Irr^{M,fm}(F')\times \Irr((\F_q^{\times})^l),$$ where the union on the left hand side is taken over all functors $F$ of the form $F=F'\oplus \bigoplus_{j=1}^r F_{c_j,c_j}$ for $c_j\notin \{a_1,b_1,\ldots, a_k,b_k\}$. 
\end{corollary}
\begin{proof}
The previous corollary shows how a representation of $\Autc(F)$ for $F=F'\oplus \bigoplus_{j=1}^r F_{c_j,c_j}$ gives a representation of $\Autc(F')\times(\F_q^{\times})^r$. We can extend this to a representation of $\Autc(F')\times (\F_q^{\times})^l$ by letting the extra factors of $\F_q^{\times}$ act trivially. In the other direction, if we have a representation $V=V'\otimes K_{(\psi_1,\ldots, \psi_l)}\in \Irr^{M,fm}(F')\times \Irr((\F_q^{\times})^l)$ then this can easily be seen to come from the functor $F= F' \bigoplus_{c:\psi_c\neq \epsilon} F_{c,c}$. 
\end{proof}

To summarise this discussion, in order to study the image of functor morphing for $B_n(\F_q)$ it is enough to study functors $F=\bigoplus F_{a_i,b_i}$ such that $a_i\neq b_i$ for all $i$, and $a_i\neq a_j$ and $b_i\neq b_j$ for $i\neq j$. We finish with examples for small values of $n$. 

When $n=2$ the only functors which do not contain singletons are $0$ and $F_{1,2}$. We cannot add singletons to $\{[1,2]\}$ and since $1\neq 2$ every irreducible representation of $\Autc(F_{1,2})\cong \F_q^{\times}$ arises by functor morphing. For the 0 functor we have $\Autc(0)=\{1\}$. We can add the singletons $[1,1]$ and $[2,2]$ and we get $\Irr^{M,fm}(0)\sqcup \Irr^{M,fm}(F_{1,1}) \sqcup \Irr^{M,fm}(F_{2,2}) \sqcup \Irr^{M,fm}(F_{1,1}\oplus F_{2,2})\cong \Irr((\F_q^{\times})^2)$. This example was given in Section 6 of \cite{CMO4}.

Consider now the case $n=3$. The only new example that we get here, that does not arise from an order preserving embedding $[1,2]\to [1,3]$, is $F_{1,2}\oplus F_{2,3}$. Indeed, the case of $F_{1,2}$, $F_{1,3}$, $F_{2,3}$ and functors that involve singletons can be deduced from the case $n=2$.  
We have $$\Autc(F_{1,2}\oplus F_{2,3}) = \{\begin{pmatrix} a & 0 \\ b & c \end{pmatrix}| a,c\in \F_q^{\times}, b\in \F_q\}.$$
An irreducible representation $V$ comes from functor morphing of $\Aut_R(M)$ if and only if the normal subgroup $N:=\{\begin{pmatrix} 1 & 0 \\ b & 1 \end{pmatrix} | b\in \F_q\}$ acts on it non-trivially. Using Clifford theory for $N$ we get $\Irr^{M,fm}(F_{1,2}\oplus F_{2,3})\cong \Irr(\F_q^{\times})$, because there is only one non-trivial $\Autc(\F_{1,2}\oplus F_{2,3})$-orbit in $\Irr(N)$, and its stabilizer in $\Autc(F_{1,2}\oplus F_{2,3})/N$ is given by $\{\begin{pmatrix} a & 0 \\ 0 & a \end{pmatrix}\}\cong \F_q^{\times}$.

For $n=4$ the only cases that do not arise from an order preserving embedding $[1,3]\to [1,4]$ are $F_{1,2}\oplus F_{3,4}$, $F_{1,3}\oplus F_{2,4}$, $F_{1,4}\oplus F_{2,3}$, $F_{1,2}\oplus F_{2,3}\oplus F_{3,4}$. In the first three cases the end points of the intervals do not intersect, and therefore functor morphing poses no restriction. We get $\Irr^{M,fm}(\Autc(F_{1,2}\oplus F_{3,4}))\cong \Irr^{M,fm}(\Autc(F_{1,4}\oplus F_{2,3}))\cong(\F_q^{\times})^2$ and $\Irr^{M,fm}(\Autc(F_{1,3}\oplus F_{2,4}))\cong B_2(\F_q)$, because there is a non-trivial morphism $F_{1,3}\to F_{2,4}$. 
By using Clifford Theory in a similar way to the case $F_{1,2}\oplus F_{2,3}$ when $n=3$ we get that $\Irr^{M,fm}(\Autc(F_{1,2}\oplus F_{2,3}\oplus F_{3,4}))\cong \Irr(\F_q^{\times})$. 
\section*{Acknowledgments}
I would like to thank my friends Tiah Thackston-Watson and Danny Kane for their hospitality during the writing of this manuscript. 



\begin{thebibliography}{abc}
\bibitem[CMO1]{CMO1} Tyrone Crisp, Ehud Meir and Uri Onn, A variant of Harish-Chandra functors. J. Inst.  Math.  Jussieu. Volume 18, Issue 5 (2019), 993-1049. arXiv:1607.04486. 
\bibitem[CMO2]{CMO2} Tyrone Crisp, Ehud Meir and Uri Onn, Principal series for general linear groups over finite commutative rings. Communications in Algebra, Volume 49, 2021 - Issue 11. arXiv:1704.05575. 
\bibitem[CMO3]{CMO3} Tyrone Crisp, Ehud Meir and Uri Onn, An inductive approach to representations of general linear groups over compact discrete valuation rings.  Advances in Mathematics 440 (2024) 109516 arXiv:2005.05553.
\bibitem[CMO4]{CMO4} Tyrone Crisp, Ehud Meir and Uri Onn, Functor morphing and representations of automorphism groups of modules. Representation Theory Volume 29, Pages 789-837 (2025). arXiv:2308.03248
\bibitem[H1]{Higman1} G. Higman, Enumerating p-groups. I. Inequalities, Proc. LMS 3 (1960), 24-30
\bibitem[M1]{meir1}Ehud Meir, Interpolations of monoidal categories and algebraic structures by invariant theory. Selecta Mathematica 29 (2023),
no. 58. arXiv:2105.04622
\bibitem[So1]{So1} Andrew Soffer, Combinatorics of conjugacy classes in $U_n(F_q)$, PhD thesis, UCLA, https://escholarship.org/uc/item/08q9k3gg
\bibitem[VA1]{VA4} A. Vera-Lopez and J. M. Arregi, Conjugacy classes in unitriangular matrices, Linear Algebra Appl. 370 (2003), 85-124.
\bibitem[Z1]{Zelevinsky} A. Zelevinsky, Representations of Finite Classical Groups: A Hopf Algebra Approach, Springer Lecture Notes in Mathematics (1981). 
\end{thebibliography}
\end{document}